\documentclass{amsart}

\usepackage{hyperref}
\usepackage{amsmath}
\usepackage{amssymb}
\usepackage{array}

\theoremstyle{definition}

\newtheorem{thm}{Theorem}[section]

 \newtheorem{prop}[thm]{Proposition}
 \theoremstyle{definition}
 \newtheorem{defn}[thm]{Definition}
 \theoremstyle{remark}
 \newtheorem{rem}[thm]{Remark}
 
 \numberwithin{equation}{section}

\def\D{\displaystyle}
\numberwithin{equation}{section}

\DeclareMathOperator{\ord}{ord} \DeclareMathOperator{\Card}{Card}
 \DeclareMathOperator{\rk}{rk}
\DeclareMathOperator{\trdeg}{tr.deg}
\DeclareMathOperator{\Ker}{Ker}

\begin{document}

\title[Dimension Polynomials and the Einstein's Strength] {Dimension
Polynomials and the Einstein's Strength of Some Systems of
Quasi-linear Algebraic Difference\\ Equations}

\thanks{This work was completed with the support of the NSF grant
CCF-1714425.}

\author{Alexander Evgrafov}
\address{Department of Analytical, Physical and Colloid Chemistry,
Sechenov First Moscow State Medical University, Moscow 119991,
Russia} \email{afkx\_farm@mail.ru}

\author{Alexander Levin}
\address{Department of Mathematics, The Catholic University of
America, Washington, D.C. 20064}

%\curraddr{Department of Mathematics and Statistics,
%Case Western Reserve University, Cleveland, Ohio 43403}
\email{levin@cua.edu}
%\thanks{The author was supported by NSF Grant CCF \#1016608.}

%    General info

\date\today% and, in revised form, June 22, 2001.}


\begin{abstract}
In this paper we present a method of characteristic sets for
inversive difference polynomials and apply it to the analysis of systems of
quasi-linear algebraic difference equations. We describe characteristic sets
and compute difference dimension polynomials associated with some such systems.
Then we apply our results to the comparative analysis of difference schemes for some
PDEs from the point of view of their Einstein's strength. In
particular, we determine the Einstein's strength of standard
finite-difference schemes for the Murray, Burgers and some other
reaction-diffusion equations.

\end{abstract}

%%% ----------------------------------------------------------------------
\maketitle
%%% ----------------------------------------------------------------------
%\tableofcontents
\section{Introduction}

Difference dimension polynomials, first introduced in \cite{Levin1}
and \cite{Levin2}, can be viewed as difference algebraic
counterparts of Hilbert polynomials in commutative algebra and
algebraic geometry, as well as of differential dimension polynomials
in differential algebra. Difference dimension polynomials and their
invariants are power tools for the study of difference and inversive
difference field extensions, systems of algebraic difference
equations, difference and inversive difference rings and modules
(see, for example, \cite[Ch. 6, 7]{KLMP}, \cite{Levin1}, \cite{Levin2},
\cite{Levin3}, \cite{LM1}, and \cite[Sect. 3.6, 4.6]{Levin4}).
Moreover, difference dimension polynomials play a significant role in
the qualitative theory of difference equations, because the difference
dimension polynomial of a system of algebraic difference equations
expresses the Einstein's strength of the system (see \cite[Chapter
7]{Levin4} for a detailed description of this concept).

In this paper we present a method of characteristic sets for
inversive difference polynomials and apply it for the computation of
difference dimension polynomials associated with difference schemes
for some systems of quasi-linear algebraic PDEs. The characteristics
sets of inversive difference polynomials were introduced in
\cite[Section 3.4]{KLMP}; their study was continued and extended to
the case of several term orderings in \cite[Section 2.3]{Levin4} and
\cite{Levin5}. A similar theory for non-inversive difference
polynomials was developed in \cite[Section 3.4]{KLMP}, \cite{GZ},
\cite{GLY}, \cite{GYZ} and some other works.

Hitherto, algorithmic methods for computing difference dimension
polynomials (and therefore for determining the Einstein's strength)
have been developed just for systems of linear difference equations.
This work provides methods of computation of difference dimension
polynomials for essentially wider class of systems of difference
equations. We prove the results on characteristic sets of difference
ideals generated by quasi-linear difference polynomials that allow
one to determine the Einstein's strength of important non-linear
systems. In particular, we determine the strengths of systems of
partial difference equations that arise from such schemes for
reaction-diffusion PDEs. These equations play the key role in the
theoretical foundation of the main methods for accurate, and rapid
determination of biologically active organic carboxylic acids in
objects such as infusion solutions and blood preservatives (see
\cite{Evgrafof} for the description of the corresponding
techniques).

\section{Preliminaries}

Throughout the paper, $\mathbb{N}, \mathbb{Z}$, $\mathbb{Q}$, and
$\mathbb{R}$ denote the sets of all non-negative integers, integers,
rational numbers, and real numbers, respectively. The number of
elements of a set $A$ is denoted by $\Card A$. As usual, $\mathbb
{Q}[t]$ denotes the ring of polynomials in one variable $t$ with
rational coefficients. All fields considered in the paper are supposed
to be of characteristic zero. Every ring homomorphism
is unitary (maps unity onto unity), every subring of a ring contains
the unity of the ring.

\smallskip

If $B = A_{1}\times\dots\times A_{k}$ is a Cartesian product of $k$
ordered sets with orders $\leq_{1},\dots \leq_{k}$, respectively
($k\in \mathbb{N}$, $k\geq 1$), then by the product order on $B$ we
mean a partial order $\leq_{P}$ such that $(a_{1},\dots,
a_{k})\leq_{P}(a'_{1},\dots, a'_{k})$ if and only if
$a_{i}\leq_{i}a'_{i}$ for $i=1,\dots, k$. In particular, if $a =
(a_{1},\dots, a_{k}), \,a'=(a'_{1},\dots, a'_{k})\in
\mathbb{N}^{k}$, then $a\leq_{P}a'$ if and only if $a_{i}\leq
a'_{i}$ for $i=1,\dots, k$. We write $a <_{P}a'$ if $a\leq_{P}a'$
and $a\neq a'$. The lexicographic order on $\mathbb{N}^{k}$ is denoted by
$\leq_{lex}$. If it is strict, we use the symbol $<_{lex}$.

\smallskip

In this section we present some background material needed for the
rest of the paper.

\smallskip

{\bf 2.1.\, Numerical polynomials of subsets of $\mathbb{N}^{m}$ and
$\mathbb{Z}^{m}$}.

\begin{defn}
A polynomial $f(t)$ in one variable $t$ with rational coefficients is called
numerical if $f(r)\in \mathbb{Z}$ for all sufficiently large $r\in\mathbb{Z}$.
\end{defn}

Of course, every polynomial with integer coefficients is numerical.
As an example of a numerical polynomial with non-integer
coefficients one can consider a polynomial $\D{t\choose k}$ \, where
$k\in\mathbb{N}$. (As usual, $\D{t\choose k}$ ($k\geq 1$) denotes
the polynomial $\D\frac{t(t-1)\dots (t-k+1)}{k!}$, $\D{t\choose0} =
1$, and $\D{t\choose k} = 0$ if $k < 0$.)

The following theorem proved in ~\cite[Chapter 0, section
17]{Kolchin} gives the ``canonical'' representation of a numerical
polynomial.

\begin{thm}
Let $f(t)$ be a numerical polynomial of degree $d$. Then $f(t)$ can
be represented in the form

\begin{equation}
f(t) =\D\sum_{i=0}^{d}a_{i}{{t+i}\choose i}
\end{equation}
with uniquely defined integer coefficients $a_{i}$.
\end{thm}

In what follows (until the end of the section), we deal with subsets of
$\mathbb{Z}^{m}$ ($m$ is a positive integer). If $a = (a_{1},\dots, a_{m})\in \mathbb{Z}^{m}$,
then the number $\sum_{i=1}^{m}|a_{i}|$ will be called the {\em order} of
the $m$-tuple $a$; it is denoted by $\ord a$. Furthermore, the
set $\mathbb{Z}^{m}$ will be considered as the union
\begin{equation}
\mathbb{Z}^{m} = \bigcup_{1\leq j\leq 2^{m}}\mathbb{Z}_{j}^{(m)}
\end{equation}
where $\mathbb{Z}_{1}^{(m)}, \dots, \mathbb{Z}_{2^{m}}^{(m)}$ are
all distinct Cartesian products of $m$ sets each of which is either
$\mathbb{N}$ or $\mathbb{Z}_{-}=\{a\in \mathbb{Z}|a\leq 0\}$. We
assume that $\mathbb{Z}_{1}^{(m)} = \mathbb{N}^{m}$ and call
$\mathbb{Z}_{j}^{(m)}$ the {\em $j$th orthant} of the set
$\mathbb{Z}^{m}$ ($1\leq j\leq 2^{m}$).

The set $\mathbb{Z}^{m}$ will be considered as a partially ordered set with
the order $\unlhd$ defined as follows: $(e_{1},\dots, e_{m})\unlhd (e'_{1},\dots, e'_{m})$ if
and only if the $m$-tuples $(e_{1},\dots, e_{m})$ and $(e'_{1},\dots, e'_{m})$ belong to the
same orthant $\mathbb{Z}_{k}^{(m)}$ and the $m$-tuple $(|e_{1}|,\dots, |e_{m}|)$
is less than $(|e'_{1}|,\dots, |e'_{m}|)$ with respect to the product order on $\mathbb{N}^{m}$.

If $A\subseteq \mathbb{Z}^{m}$, then $W_{A}$
will denote the set of all elements of $\mathbb{Z}^{m}$ that do not
exceed any element of $A$ with respect to the order $\unlhd$. (Thus,
$w\in W_{A}$ if and only if there is no $a\in A$ such that $a\unlhd
w$.) Furthermore, for any $r\in \mathbb{N}$, $A(r)$ will denote the
set of all elements $x = (x_{1},\dots, x_{m})\in A$ such that $\ord
x\leq r$.

\smallskip

The above notation can be naturally applied to subsets of
$\mathbb{N}^{m}$ (treated as subsets of $\mathbb{Z}^{m}$). If
$E\subseteq \mathbb{N}^{m}$ and $s\in \mathbb{N}$, then $E(s)$ will
denote the set of all $m$-tuples $e = (e_{1},\dots, e_{m})\in E$
such that $\ord e\leq s$. Furthermore, we shall associate with a set
$E\subseteq \mathbb{N}^{m}$ the set $V_{E}\subseteq \mathbb{N}^{m}$
that consists of all $m$-tuples $v = (v_{1},\dots , v_{m})\in
\mathbb{N}^{m}$ that are not greater than or equal to any $m$-tuple
from $E$ with respect to the product order on $\mathbb{N}^{m}$.
(Clearly, an element $v=(v_{1}, \dots , v_{m})\in \mathbb{N}^{m}$
belongs to $V_{E}$ if and only if for any element  $(e_{1},\dots ,
e_{m})\in E$, there exists $i\in \mathbb{N}, 1\leq i\leq m$, such
that $e_{i} > v_{i}$.)

The following two theorems proved, respectively, in \cite[Chapter 0,
section 17]{Kolchin} and \cite[Chapter 2]{KLMP} introduce certain
numerical polynomials associated with subsets of $\mathbb{N}^{m}$
and give explicit formulas for the computation of these polynomials.

\begin{thm}
Let $E$ be a subset of $\mathbb{N}^{m}$. Then there exists a
numerical polynomial $\omega_{E}(t)$ with the following properties:

\smallskip

{\em (i)} \,  $\omega_{E}(r) = \Card V_{E}(r)$ for all sufficiently
large $r\in \mathbb{N}$.

\smallskip

{\em (ii)} \, $\deg\omega_{E}$ does not exceed $m$ and
$\deg\omega_{E} = m$ if and only if $E=\emptyset$. In the last case,
$\omega_{E}(t) = \D{t+m\choose m}$.
\end{thm}

The polynomial $\omega_{E}(t)$ is called the {\em dimension
polynomial} of the set $E\subseteq \mathbb{N}^{m}$ .

\begin{thm}
Let $E = \{e_{1}, \dots, e_{q}\}$ ($q\geq 1$) be a finite subset of
$\mathbb{N}^{m}$.  Let $e_{i} = (e_{i1}, \dots, e_{im})$ \, ($1\leq
i\leq q$) and for any $l\in \mathbb{N}$, $0\leq l\leq q$, let
$\Theta(l,q)$ denote the set of all $l$-element subsets of the set
$\mathbb{N}_{q} = \{1,\dots, q\}$. Furthermore, let
$\bar{e}_{\emptyset j} = 0$ and for any $\theta\in\Theta(l,q)$,
$\theta\neq \emptyset$, let $\bar{e}_{\theta j} = \max
\{e_{ij}\,|\,i\in \theta\}$, $1\leq j\leq m$. (In other words, if
$\theta = \{i_{1},\dots, i_{l}\}$, then $\bar{e}_{\theta j}$ denotes
the greatest $j$th coordinate of the elements $e_{i_{1}},\dots,
e_{i_{l}}$.) Furthermore, let  $b_{\theta}
=\D\sum_{j=1}^{m}\bar{e}_{\theta j}$. Then
\begin{equation}\label{eq:6} \omega_{E}(t) =
\D\sum_{l=0}^{q}(-1)^{l}\D\sum_{\theta\in\Theta(l,\, q)}{t+ m -
b_{\theta}\choose m}
\end{equation}
\end{thm}

{\bf Remark.} \, Clearly, if $E\subseteq \mathbb{N}^{n}$ and
$E^{\ast}$ is the set of all minimal elements of the set $E$ with
respect to the product order on $\mathbb{N}^{m}$, then the set
$E^{\ast}$ is finite and $\omega_{E}(t) = \omega_{E^{\ast}}(t)$.
Thus, the last theorem gives an algorithm that allows one to find a
numerical polynomial associated with any subset of $\mathbb{N}^{m}$:
one should first find the set of all minimal points of the subset
and then apply Theorem 2.4.

\medskip

The following two results, proved in \cite[Section 2.5]{KLMP},
describe dimension polynomials associated with subsets of
$\mathbb{Z}^{m}$.

\begin{thm}
Let $A$ be a subset of $\mathbb{Z}^{m}$. Then there exists a
numerical polynomial $\phi_{A}(t)$ such that

\smallskip

{\em (i)}\, $\phi_{A}(r) = \Card W_{A}(r)$ for all sufficiently
large $r\in\mathbb{N}$.

\smallskip

{\em (ii)}\, $\deg\phi_{A}\leq m$ and the polynomial $\phi_{A}(t)$
can be written in the form $\phi_{A}(t) =
\D\sum_{i=0}^{m}a_{i}{{t+i}\choose i}$ where $a_{i}\in\mathbb{Z}$
and $2^{m}|a_{m}$.

\smallskip

{\em (iii)}\, $\phi_{A}(t)=0$ if and only if $(0,\dots,0)\in A$.

\smallskip

{\em (iv)}\, If $A = \emptyset$, then $\phi_{A}(t) =
\D\sum_{i=0}^{m}(-1)^{m-i}2^{i}{m\choose i}{{t+i}\choose i}$.
\end{thm}

\begin{thm}
With the notation of Theorem 2.5, let us consider a mapping $\rho:
\mathbb{Z}^{m}\longrightarrow\mathbb{N}^{2m}$ such that
$$\rho((e_{1},\dots, e_{m})) =(\max \{e_{1}, 0\}, \dots, \max \{e_{m}, 0\}, \max \{-e_{1}, 0\}, \dots, \max \{-e_{m}, 0\}).$$
Let $B = \rho(A)\bigcup \{\bar{e}_{1},\dots, \bar{e}_{m}\}$ where
$\bar{e}_{i}$ ($1\leq i\leq m$) is a $2m$-tuple in $\mathbb{N}^{2m}$
whose $i$th and $(m+i)$th coordinates are equal to 1 and all other
coordinates are equal to 0. Then $\phi_{A}(t)= \omega_{B}(t)$ where
$\omega_{B}(t)$ is the dimension polynomial of the set $B\subseteq
\mathbb{N}^{2m}$ (i. e., the dimension polynomial introduced in
Theorem 2.3).
\end{thm}
The polynomial $\phi_{A}(t)$ is called the {\em dimension
polynomial} of the set $A\subseteq \mathbb{Z}^{m}$. It is easy to
see that Theorems 2.6 and 2.4 provide an algorithm for computing
such a polynomial.

\bigskip

{\bf 2.2.\, Some basic facts from difference algebra.}\, A {\em
difference ring} is a commutative ring $R$ together with a finite
set $\sigma = \{\alpha_{1}, \dots, \alpha_{m}\}$ of mutually
commuting injective endomorphisms of $R$ into itself. The set
$\sigma$ is called the {\em basic set\/} of the difference ring $R$,
and the endomorphisms $\alpha_{1}, \dots, \alpha_{m}$ are called
{\em translations.}\, A difference ring with a basic set $\sigma$ is
also called a {\em $\sigma$-ring}. If $\alpha_{1}, \dots,
\alpha_{m}$ are automorphisms of $R$, we say that $R$ is an {\em
inversive difference ring\/} with the basic set $\sigma$.  In this
case we denote the set $\{\alpha_{1}, \dots, \alpha_{m},
\alpha_{1}^{-1}, \dots, \alpha_{m}^{-1}\}$ by $\sigma^{\ast}$ and
call $R$ a $\sigma^{\ast}$-ring. If a difference ($\sigma$-) ring
$R$ is a field, it is called a difference (or $\sigma$-) field. If
$R$ is inversive, it is called an inversive difference field or a
$\sigma^{\ast}$-field. In what follows we deal only with inversive ($\sigma^{\ast}$-) rings and fields
where $\sigma = \{\alpha_{1}, \dots, \alpha_{m}\}$.

If $R$ is a $\sigma^{\ast}$-ring and $R_{0}$ a subring of $R$ such
that $\alpha(R_{0})\subseteq R_{0}$ for any $\alpha \in
\sigma^{\ast}$, then $R_{0}$ is called a $\sigma^{\ast}$-{\em
subring\/} of $R$, while the ring $R$ is said to be a
$\sigma^{\ast}$-{\em overring} of $R_{0}$. In this case the
restriction of an endomorphism $\alpha_{i}$ on $R_{0}$ is denoted by
the same symbol $\alpha_{i}$. If $R$ is a $\sigma^{\ast}$-field and
$R_{0}$ a subfield of $R$, which is also a $\sigma^{\ast}$-subring
of $R$, then  $R_{0}$ is said to be a $\sigma^{\ast}$-subfield of
$R$; $R$, in turn, is called a $\sigma^{\ast}$-field extension or a
$\sigma$-overfield of $R_{0}$. In this case we also say that we have
a $\sigma^{\ast}$-field extension $R/R_{0}$.

If $R$ is an inversive difference ring with a basic set $\sigma$ and
$J$ is an ideal of $R$ such that $\alpha(J)\subseteq J$ for any
$\alpha \in \sigma^{\ast}$, then $J$ is called a {\em
$\sigma^{\ast}$-ideal} of $R$. If a $\sigma^{\ast}$-ideal $P$ of the
ring $R$ is prime (in the usual sense), we say that $P$ is a {\em
prime} $\sigma^{\ast}$-{\em ideal} of $R$. An element $a\in R$ is
said to be a {\em constant} if $\alpha(a) = a$ for every
$\alpha\in\sigma$.

\smallskip

If $R$ is a $\sigma^{\ast}$-ring, then $\Gamma$ will denote the free commutative group
generated by the set $\sigma$. Elements of the
group $\Gamma$ (written in the multiplicative form
$\alpha_{1}^{i_{1}}\dots \alpha_{m}^{i_{m}}$ with $i_{1},\dots,
i_{m}\in \mathbb{Z}$) act on $R$ as automorphisms that are compositions of the automorphisms
from the set $\sigma^{\ast}$.

If $S$ is a subset of a $\sigma^{\ast}$-ring $R$, then $[S]^{\ast}$
will denote the smallest $\sigma^{\ast}$-ideal of $R$ containing
$S$; as an ideal, it is generated by the set $\Gamma S = \{\gamma(a)
| \gamma \in \Gamma, a\in S\}$. If $S$ is finite, $S=\{a_{1},\dots,
a_{k}\}$, we write $[a_{1},\dots, a_{k}]^{\ast}$ for $I =
[S]^{\ast}$ and say that $I$ is a finitely generated
$\sigma^{\ast}$-ideal of $R$. (In this case, elements $a_{1},\dots,
a_{k}$ are said to be $\sigma^{\ast}$-generators of $I$.) If $R_{0}$
is a $\sigma^{\ast}$-subring of $R$, then the intersection of all
$\sigma^{\ast}$-subrings of $R$ containing $R_{0}$ and a set
$B\subseteq R$ is the smallest $\sigma^{\ast}$-subring of $R$
containing $R_{0}$ and $B$. This ring coincides with the ring
$R_{0}[\{\gamma(b)\,|\,b\in B, \gamma \in \Gamma\}]$; it is denoted
by $R_{0}\{B\}^{\ast}$. The set $B$ is said to be a {\em set of
$\sigma^{\ast}$-generators} of $R_{0}\{B\}^{\ast}$ over $R_{0}$. If
$B = \{b_{1},\dots, b_{k}\}$ is a finite set, we say that $R_{1} =
R_{0}\{B\}^{\ast}$ is a finitely generated inversive difference (or
$\sigma^{\ast}$-) ring extension (or overring) of $R_{0}$ and write
$R_{1} = R_{0}\{b_{1},\dots, b_{k}\}^{\ast}$.

If $R$ is a $\sigma^{\ast}$-field, $R_{0}$ a
$\sigma^{\ast}$-subfield of $R$ and $B\subseteq R$, then the
intersection of all $\sigma^{\ast}$-subfields of $R$ containing
$R_{0}$ and $B$ is denoted by $R_{0}\langle B\rangle^{\ast}$. This
is the smallest $\sigma^{\ast}$-subfield of $R$ containing $R_{0}$
and $B$; it coincides with the field $R_{0}(\{\gamma(b) | b\in B,
\gamma \in \Gamma\})$. The set $B$ is called a {\em set of
$\sigma^{\ast}$-generators of the $\sigma^{\ast}$-field extension}
$R_{0}\langle B\rangle^{\ast}$ {\em of} $R_{0}$.  If $B$ is finite,
$B=\{b_{1},\dots, b_{k}\}$, we write $R_{0}\langle b_{1},\dots,
b_{k}\rangle^{\ast}$ for $R_{0}\langle B\rangle^{\ast}$.

In what follows we often consider two or more inversive difference rings
$R_{1},\dots, R_{p}$ with the same basic set
$\sigma=\{\alpha_{1},\dots, \alpha_{m}\}$. Formally speaking, it
means that for every $i=1,\dots, p$, there is some fixed mapping
$\nu_{i}$ from the set $\sigma$ into the set of all injective
endomorphisms of the ring $R_{i}$ such that any two endomorphisms
$\nu_{i}(\alpha_{j})$ and $\nu_{i}(\alpha_{k})$ of $R_{i}$ commute
($1\leq j, k\leq n$). We shall identify elements $\alpha_{j}$ with
their images $\nu_{i}(\alpha_{j})$ and say that elements of the set
$\sigma$ act as mutually commuting automorphisms of the
ring $R_{i}$ ($i=1,\dots, p$).

Let $R_{1}$ and $R_{2}$ be inversive difference rings with the same basic set
$\sigma=\{\alpha_{1},\dots, \alpha_{m}\}$. A ring homomorphism
$\phi: R_{1}\rightarrow R_{2}$ is called a {\em difference} (or
$\sigma$-) {\em homomorphism} if $\phi(\alpha(a)) =
\alpha(\phi(a))$ for any $\alpha \in \sigma, a\in R_{1}$. Clearly,
if $\phi: R_{1} \rightarrow R_{2}$ is a $\sigma$-homomorphism of
inversive difference rings, then $\phi(\alpha^{-1}(a)) =
\alpha^{-1}(\phi(a))$ for any $\alpha \in \sigma, \, a\in R_{1}$. If
a $\sigma$-homomorphism is an isomorphism (endomorphism,
automorphism, etc.), it is called a difference (or $\sigma$-)
isomorphism (respectively, difference (or $\sigma$-) endomorphism,
difference (or $\sigma$-) automorphism, etc.).  If $R_{1}$ and
$R_{2}$ are two $\sigma^{\ast}$-overrings of the same $\sigma^{\ast}$-ring $R_{0}$
and $\phi: R_{1}\rightarrow R_{2}$ is a $\sigma$-homomorphism such
that $\phi(a) = a$ for any $a\in R_{0}$, we say that $\phi$ is a
difference (or $\sigma$-) homomorphism over $R_{0}$ or that $\phi$
leaves the ring $R_{0}$ fixed.

It is easy to see that the kernel of any $\sigma$-homomorphism of
$\sigma^{\ast}$-rings $\phi: R\rightarrow R'$ is a $\sigma^{\ast}$-ideal
of $R$. Conversely, let $g$ be a surjective
homomorphism of a $\sigma^{\ast}$-ring $R$ onto a ring $S$ such that $\Ker
g$ is a $\sigma^{\ast}$-ideal of $R$. Then there is a unique
structure of a $\sigma^{\ast}$-ring on $S$ such that $g$ is a
$\sigma$-homomorphism. In particular, if $I$ is a
$\sigma^{\ast}$-ideal of a $\sigma^{\ast}$-ring $R$, then the factor ring
$R/I$ has a unique structure of a $\sigma^{\ast}$-ring such that the
canonical surjection $R\rightarrow R/I$ is a $\sigma$-homomorphism.
In this case $R/I$ is said to be the $\sigma^{\ast}$-{\em factor ring}
of $R$ by the $\sigma^{\ast}$-ideal $I$.

If a $\sigma^{\ast}$-ring $R$ is an integral domain, then its
quotient field $Q(R)$ can be naturally considered as a
$\sigma^{\ast}$-overring of $R$. In this case $Q(R)$ is said to be
the quotient $\sigma^{\ast}$-{\em field} of $R$. Clearly, if a
$\sigma^{\ast}$-field $K$ contains $R$ as a $\sigma^{\ast}$-subring,
then $K$ contains the quotient $\sigma^{\ast}$-field $Q(R)$.

Let $R$ be a $\sigma^{\ast}$-ring, $\Gamma$ the free commutative
group generated by $\sigma$, and $Y = \{y_{i}\,|\,i\in I\}$ a family
of elements from some $\sigma^{\ast}$-overring of $R$. We say that
the family $Y$ is {\em transformally} (or $\sigma$-{\em
algebraically) dependent} over $R$, if the family $\Gamma(Y) =
\{\gamma(y_{i})\,|\,i\in I, i\in I\}$ is algebraically dependent
over $R$ (that is, there exist elements $v_{1},\dots, v_{k}\in
\Gamma(Y)$ and a non-zero polynomial $f(X_{1},\dots, X_{k})$ with
coefficients in $R$ such that $f(v_{1},\dots, v_{k}) = 0$).
Otherwise, the family $Y$ is said to be {\em transformally} (or
$\sigma$-{\em algebraically) independent} over $R$ or a family of
{\em inversive difference} (or $\sigma^{\ast}$-) {\em
indeterminates} over $R$. In the last case, the $\sigma^{\ast}$-ring
$R\{(y_{i})_{i\in I}\}^{\ast}$ is called the {\em algebra of
inversive difference} (or $\sigma^{\ast}$-) {\em polynomials} over
$R$. As it is shown in \cite[Proposition 3.4.4]{KLMP}, for any set
$I$, there exists an algebra of $\sigma^{\ast}$-polynomials $S =
R\{(y_{i})_{i\in I}\}^{\ast}$ over $R$ in a family of
$\sigma$-indeterminates $Y = \{y_{i}\,|\,i\in I\}$ with indices from
the set $I$. If $S$ and $S'$ are two such algebras, then there
exists a $\sigma$-isomorphism $S\rightarrow S'$ that leaves the ring
$R$ fixed. If $R$ is an integral domain, then any algebra of
$\sigma^{\ast}$-polynomials over $R$ is an integral domain.

The algebra of $\sigma^{\ast}$-polynomials in the family of
$\sigma$-indeterminates $Y$ over $R$ can be constructed by extending
the natural structure of a $\sigma^{\ast}$-ring from $R$ to the
polynomial ring $S = R[\{y_{i,\gamma}\,|\,i\in I,
\gamma\in\Gamma\}]$ in the set of indeterminates $\{y_{i,\gamma}\}$
indexed by $I\times\Gamma$. The extension of the action of an
element $\beta\in\sigma^{\ast}$ from $R$ to $S$ is defined by
$\beta(y_{i,\gamma}) = y_{i,\beta\gamma}$ ($i\in I,
\gamma\in\Gamma$); in what follows, we denote $y_{i,1}$ by $y_{i}$
and write $\gamma y_{i}$ for $y_{i,\gamma}$.
\begin{rem} Power products of elements of $\sigma$ with nonnegative exponents form a
commutative semigroup $T\subset\Gamma$; in the case of non-inversive difference rings,
a family $Y = \{y_{i}\,|\,i\in I\}$ is said to be $\sigma$-algebraically
independent over a $\sigma$-ring $R$ if the family
$T(Y) = \{\tau y_{i}\,|\,\tau\in T, i\in I\}$ is algebraically independent over $R$.
If $R$ is a  $\sigma^{\ast}$-ring and $Y$ is a family of elements of some
$\sigma^{\ast}$-overring of $R$, then the family $\Gamma(Y)$
is algebraically dependent over $R$ if and only if the family $T(Y)$ has this property.
That is why we use the term ''$\sigma$-algebraically dependent'',
not ''$\sigma^{\ast}$-algebraically dependent''.
\end{rem}
Let $K$ be a $\sigma^{\ast}$-field and $L$ a
$\sigma^{\ast}$-overfield of $K$. An element $u\in L$ is said to be
{\em transformally algebraic} (or {\em $\sigma$-algebraic}) if the
family $\{\gamma u\,|\,\gamma\in\Gamma\}$ is algebraic over $K$.
Otherwise, we say that $u$ is {\em transformally} (or $\sigma$-)
{\em transcendental} over $K$. As it is shown in \cite[Sect.
4.1]{Levin4}, there is a subset $B$ of $L$ such that $B$ is
$\sigma$-algebraically independent over $K$ and every element of $L$
is $\sigma$-algebraic over $K\langle B\rangle^{\ast}$. Such a set
$B$ is called a $\sigma$-transcendence basis of $L$ over $K$. All
$\sigma$-transcendence bases of $L$ over $K$ have the same
cardinality. If the $\sigma^{\ast}$-field extension $L/K$ is
finitely generated, then every $\sigma$-transcendence basis of $L$
over $K$ is finite; the number of its elements is called the {\em
$\sigma$-transcendence degree} of the extension $L/K$ and is denoted
by $\sigma$-$\trdeg_{K}L$.

If $K\{(y_{i})_{i\in I}\}^{\ast}$ is an algebra of
$\sigma^{\ast}$-polynomials over a $\sigma^{\ast}$-ring $K$ and
$(\eta_{i})_{i\in I}$ a family of elements from a
$\sigma^{\ast}$-overfield of $K$, one can define a surjective
$\sigma$-homomorphism $\phi_{\eta}:\,K\{(y_{i})_{i\in
I}\}^{\ast}\rightarrow K\{(\eta_{i})_{i\in I}\}^{\ast}$ that maps
every $y_{i}$ onto $\eta_{i}$ and leaves elements of $K$ fixed. This
homomorphism is called the substitution of $(\eta_{i})_{i\in I}$ for
$(y_{i})_{i\in I}$. If $g$ is a $\sigma^{\ast}$- polynomial, then
its image under a substitution of $(\eta_{i})_{i\in I}$ for
$(y_{i})_{i\in I}$ is denoted by $g((\eta_{i})_{i\in I})$. The
kernel of $P$ of $\phi_{\eta}$ is a prime $\sigma^{\ast}$-ideal of
$K\{(y_{i})_{i\in I}\}^{\ast}$, since $K\{(\eta_{i})_{i\in
I}\}^{\ast}$ is an integral domain (it is contained in the field
$L$). Therefore, the $\sigma^{\ast}$-field $K\langle(\eta_{i})_{i\in
I}\rangle^{\ast}$ can be treated as the quotient
$\sigma^{\ast}$-field of $R\{(y_{i})_{i\in I}\}^{\ast}/P$.

Let $K$ be an inversive difference ($\sigma^{\ast}$-) field and $n$ a
positive integer. By an {\em $n$-tuple over} $K$ we mean an
$n$-dimensional vector $a = (a_{1},\dots, a_{n})$ whose coordinates
belong to some $\sigma^{\ast}$-overfield of $K$.  If each $a_{i}$
($1\leq i\leq n$) is $\sigma$-algebraic over the $\sigma$-field $K$,
we say that the {\em $n$-tuple $a$ is $\sigma$-algebraic over $K$}.

Let $R = K\{y_{1},\dots, y_{n}\}^{\ast}$ be the algebra of
$\sigma^{\ast}$-polynomials in $n$ $\sigma^{\ast}$-indeterminates
$y_{1},\dots, y_{n}$ over $K$ and  $\Phi = \{f_{j} | j\in
J\}\subseteq R$. An $n$-tuple $\eta = (\eta_{1},\dots, \eta_{n})$
over $K$ is said to be a solution of the set $\Phi$ or a solution of
the {\em system of algebraic difference equations}
$f_{j}(y_{1},\dots, y_{n}) = 0$ ($j\in J$) if $\Phi$ is contained in
the kernel of the substitution of $(\eta_{1},\dots, \eta_{n})$ for
$(y_{1},\dots, y_{n})$. A system of algebraic difference equations
$\Phi$ is said to be {\em prime} if the $\sigma^{\ast}$-ideal
generated by $\Phi$ in the ring $R$ is prime.

Clearly, if one fixes an $n$-tuple $\eta = (\eta_{1},\dots,
\eta_{n})$ over a $\sigma^{\ast}$-field $K$, then all
$\sigma^{\ast}$-polynomials of the ring $R = K\{y_{1},\dots,
y_{n}\}^{\ast}$, for which $\eta$ is a solution, form a prime
$\sigma^{\ast}$-ideal; it is called the {\em defining
$\sigma^{\ast}$-ideal} of $\eta$ over $K$.

\section{Characteristic Sets and Difference Dimension Polynomials}

Let $K$ be an inversive difference ($\sigma^{\ast}$-) field with a
basic set $\sigma = \{\alpha_{1},\dots, \alpha_{m}\}$ and let
$\Gamma$ denote the free commutative group generated by $\sigma$. If
$\gamma = \alpha_{1}^{k_{1}}\dots \alpha_{m}^{k_{m}}\in \Gamma$
($k_{1},\dots, k_{m}\in \mathbb{Z}$), the {\em order} of the element
$\gamma$ is defined as $\ord\gamma = \D\sum_{i=1}^{m}|k_{i}|$. For
any $r\in\mathbb{N}$, we set $\Gamma(r) = \{\gamma\in\Gamma\,|\,
\ord\gamma\leq r\}$. Furthermore, for every $j=1,\dots, 2^{m}$, we
set $\Gamma_{j} = \{\gamma = \alpha_{1}^{k_{1}}\dots
\alpha_{m}^{k_{m}}\in\Gamma\,|\,(k_{1},\dots, k_{m})\in
\mathbb{Z}^{(m)}_{j}\}$ (see the representation (2.2) of the set
$\mathbb{Z}^{m}$ as the union of the orthants).

Let $K\{y_{1},\dots, y_{n}\}^{\ast}$ be the algebra of
$\sigma^{\ast}$-polynomials in $\sigma^{\ast}$-indeterminates
$y_{1},\dots, y_{n}$ over $K$ and let $\Gamma Y$ denote the set
$\{\gamma y_{i} | \gamma \in \Gamma, 1\leq i\leq m \}$ whose
elements are called {\em terms}. By the order of a term $u = \gamma
y_{j}$ we mean the order of the element $\gamma \in \Gamma$.
Furthermore, setting $(\Gamma Y)_{j} = \{\gamma y_{i} | \gamma \in
\Gamma_{j}, 1\leq i\leq n \}$ ($j=1,\dots, 2^{m}$) we obtain a
representation of the set of terms as a union $\Gamma Y =
\D\bigcup_{j=1}^{2^{m}}(\Gamma Y)_{j}.$
\begin{defn} A term $v\in \Gamma Y$ is called a {\bf transform}
of a term $u\in \Gamma Y$ if and only if $u$ and $v$ belong to the
same set $(\Gamma Y)_{j} \, (1\leq j\leq 2^{m})$ and $v = \gamma u$
for some $\gamma \in \Gamma_{j}$. If $\gamma \neq 1$, $v$ is said to
be a proper transform of $u$.
\end{defn}
In what follows, we say that an element $\gamma\in\Gamma$ is {\em
similar} to a term $u\in\Gamma Y$ and write $\gamma\sim u$ if
$\gamma\in\Gamma_{j}$ and $u\in(\Gamma Y)_{j}$ for the same index
$j$ ($1\leq j\leq 2^{m}$). We also write $\gamma\sim\gamma'$ if
$\gamma, \gamma'\in\Gamma_{j}$ and $u\sim v$ if $u, v\in(\Gamma
Y)_{j}$ for the same $j$.

\begin{defn} A well-ordering of the set of terms
$\Gamma Y$ is called a {\bf ranking} of the family of
$\sigma^{\ast}$-indeterminates $y_{1},\dots, y_{n}$ (or a ranking of
the set $\Gamma Y$) if it satisfies the following conditions. {\em
(}We use the standard symbol $\leq$ for the ranking; it will be
always clear what order is denoted by this symbol.{\em )}

\smallskip

{\em (i)}\, If $u\in (\Gamma Y)_{j}$ and $\gamma \in \Gamma_{j}$
($1\leq j\leq 2^{m}$), then $u\leq \gamma u$.

\smallskip

{\em (ii)}\, If $u, v\in (\Gamma Y)_{j}$ ($1\leq j\leq 2^{m}$),
$u\leq v$ and $\gamma \in \Gamma_{j}$, then $\gamma u \leq \gamma
v$.

A ranking of the $\sigma^{\ast}$-indeterminates $y_{1},\dots, y_{n}$
is called orderly if for any $j=1,\dots, 2^{m}$ and for any two
terms $u, v \in (\Gamma Y)_{j}$, the inequality $\ord\,u < \ord\,v$
implies that $u < v$ (as usual, $v < w$ means $v\leq w$ and $v\neq
w$).
\end{defn}

As an example of an orderly ranking of the
$\sigma^{\ast}$-indeterminates $y_{1},\dots, y_{n}$ one can consider
the {\em standard ranking} defined as follows: $u =
\alpha_{1}^{k_{1}}\dots \alpha_{m}^{k_{m}}y_{i}\leq v =
\alpha_{1}^{l_{1}}\dots \alpha_{m}^{l_{m}}y_{j}$ if and only if the
$(2m+2)$-tuple $(\D\sum_{\nu = 1}^{m}|k_{\nu}|, |k_{1}|,\dots,
|k_{m}|, k_{1},\dots, k_{m}, i)$ is less than or equal to the
$(2m+2)$-tuple $(\D\sum_{\nu = 1}^{m}|l_{\nu}|, |l_{1}|,\dots,
|l_{m}|, l_{1},\dots, l_{m}, j)$ with respect to the lexicographic
order on $\mathbb{Z}^{2m+2}$.

In what follows, we assume that an orderly ranking $\leq$ of the set
of $\sigma^{\ast}$-indeterminates $y_{1},\dots, y_{n}$ is fixed. If
$A\in K\{y_{1},\dots, y_{n}\}^{\ast}$, then the greatest (with
respect to the ranking $\leq$) term that appears in $A$ is called
the {\em leader} of $A$; it is denoted by $u_{A}$. If $u = u_{A}$
and $d=\deg_{u}A$, then the $\sigma^{\ast}$-polynomial $A$ can be
written as $A = I_{d}u^{d} + I_{d-1}u^{d-1}+\dots + I_{0}$ where
$I_{k} (0\leq k\leq d)$ do not contain $u$. The
$\sigma^{\ast}$-polynomial $I_{d}$ is called the {\em initial} of
$A$; it is denoted by $I_{A}$.

\begin{defn}
Let $A, B\in K\{y_{1}\dots, y_{n}\}^{\ast}$. We say that $A$ has higher
rank than $B$ and write $\rk A > \rk B$ if either $A\notin K,\, B\in
K$, or $u_{A}$ has higher rank than $u_{B}$, or $u_{A} = u_{B}$ and
$\deg_{u_{A}}A > \deg_{u_{A}}B$. If $u_{A} = u_{B}$ and
$\deg_{u_{A}}A = \deg_{u_{A}}B$, we say that $A$ and $B$ have the
same rank and write $\rk A = \rk B$.
\end{defn}

Note that distinct $\sigma^{\ast}$-polynomials can have the same
rank and if $A\notin K$, then $I_{A}$ has lower rank than $A$.

\begin{defn}
Let $A, B\in K\{y_{1},\dots, y_{n}\}^{\ast}$. The
$\sigma^{\ast}$-polynomial $A$ is said to be {\em reduced} with
respect to $B$ if $A$ does not contain any power of a transform
$\gamma u_{B}$ ($\gamma \in \Gamma$) whose exponent is greater than
or equal to $\deg_{u_{B}}B$ (recall that by the definition of a
transform, $\gamma\sim u_{B}$). If $\mathcal{A}\subseteq
K\{y_{1},\dots, y_{n}\}^{\ast}\setminus K$, then a
$\sigma^{\ast}$-polynomial $A\in K\{y_{1},\dots, y_{n}\}^{\ast}$, is
said to be reduced with respect to $\mathcal{A}$ if $A$ is reduced
with respect to every element of the set $\mathcal{A}$.

\smallskip

A set $\mathcal{A}\subseteq K\{y_{1},\dots, y_{n}\}^{\ast}$ is said
to be autoreduced if either it is empty or $\mathcal{A}\bigcap K =
\emptyset$ and every element of $\mathcal{A}$ is reduced with
respect to all other elements of the set $\mathcal{A}$.
\end{defn}

The proof of the following proposition can be obtained by mimicking
the proof of the corresponding statement about autoreduced sets of
differential polynomials, see \cite[Ch. 1, Sect. 9]{Kolchin}.

\begin{prop}
Every autoreduced set is finite and distinct elements of an
autoreduced set have distinct leaders.
\end{prop}
\begin{thm}{\em (\cite[Theorem 2.4.7]{Levin4})} Let $\mathcal{A} =
\{A_{1},\dots, A_{p}\}$ be an autoreduced subset in the ring of
$\sigma^{\ast}$-polynomials  $K\{y_{1},\dots, y_{n}\}^{\ast}$ and
let $D\in K\{y_{1},\dots, y_{n}\}^{\ast}$. Furthermore, let
$I(\mathcal{A})$ denote the set of all $\sigma^{\ast}$-polynomials
$B\in K\{y_{1},\dots, y_{n}\}^{\ast}$ such that either $B =1$ or $B$ is a
product of finitely many polynomials of the form $\gamma(I_{A_{i}})$
where $\gamma \in \Gamma,\, i=1,\dots, p$. Then there exist
$\sigma^{\ast}$-polynomials $J\in I(\mathcal{A})$ and $D_{0}\in
K\{y_{1},\dots, y_{n}\}^{\ast}$ such that $D_{0}$ is reduced with respect
to $\mathcal{A}$ and $JD\equiv D_{0} (mod\, [\mathcal{A}])$.
\end{thm}

Note that, with the notation of the last theorem, the process of
reduction that leads to the $\sigma^{\ast}$-polynomials $J\in
I(\mathcal{A})$ and $D_{0}$ is algorithmic; the steps of the
corresponding algorithm are similar to the steps described in the
proof of Theorem 2.4.1 of \cite{Levin4}. The
$\sigma^{\ast}$-polynomial $D_{0}$ is called the {\em remainder} of
$D$ with respect to $\mathcal{A}$. We also say that $D$ {\em reduces
to $D_{0}$ modulo $\mathcal{A}$}.

In what follows elements of an autoreduced set in $K\{y_{1},\dots,
y_{n}\}^{\ast}$ will be always written in the order of increasing
rank. With this assumption we introduce the following partial order
on the set of all autoreduced sets.

\begin{defn}
Let $\mathcal{A} = \{A_{1},\dots, A_{p}\}$ and $\mathcal{B} =
\{B_{1},\dots, B_{q}\}$ be two autoreduced sets of
$\sigma^{\ast}$-polynomials in $K\{y_{1},\dots, y_{n}\}^{\ast}$. We
say that $\mathcal{A}$ has lower rank than $\mathcal{B}$ and write
$\rk\mathcal{A} < \rk\mathcal{B}$ if either there exists
$k\in\mathbb{N},\, 1\leq k\leq \min\{p, q\}$, such that $\rk A_{i} =
\rk B_{i}$ for $i=1,\dots, k-1$ and $\rk A_{k} < \rk B_{k}$, or $p >
q$ and $\rk A_{i} = \rk B_{i}$ for $i=1,\dots, q$.
\end{defn}

Mimicking the arguments of \cite[Ch. 1, Sect. 9]{Kolchin}, one
obtains that every nonempty family of autoreduced subsets of
$K\{y_{1},\dots, y_{n}\}^{\ast}$ contains an autoreduced set of
lowest rank. In particular, if $\emptyset\neq J\subseteq
F\{y_{1},\dots, y_{n}\}^{\ast}$, then the set $J$ contains an
autoreduced set of lowest rank called a {\bf characteristic set} of
$J$.

\begin{prop}{\em (\cite[Proposition 2.4.8]{Levin4})}
Let $K$ be an inversive difference field with a basic set $\sigma$,
$J$ a $\sigma^{\ast}$-ideal of the algebra of
$\sigma^{\ast}$-polynomials $K\{y_{1},\dots, y_{n}\}^{\ast}$, and
$\mathcal{A}$ a characteristic set of $J$. Then

\smallskip

{\em (i)}\, The ideal $J$ does not contain nonzero
$\sigma^{\ast}$-polynomials reduced with respect to $\mathcal{A}$.
In particular, if $A\in\mathcal{A}$, then $I_{A}\notin J$.

\smallskip

{\em (ii)}\, If $J$ is a prime $\sigma^{\ast}$-ideal, then $J =
[\mathcal{A}]^{\ast}:\Upsilon(\mathcal{A})$ where
$\Upsilon(\mathcal{A})$ denotes the set of all finite products of
elements of the form $\gamma(I_{A})$ ($\gamma \in \Gamma,
A\in\mathcal{A}$).
\end{prop}

A $\sigma^{\ast}$-ideal of the ring of $\sigma^{\ast}$-polynomials
$K\{y_{1},\dots, y_{n}\}^{\ast}$ is called {\em linear} if it is
generated (as a $\sigma^{\ast}$-ideal) by homogeneous linear
$\sigma^{\ast}$-polynomials, i. e., $\sigma^{\ast}$-polynomials of
the form $\D\sum_{i=1}^{p}a_{i}\gamma_{i}y_{k_{i}}$ ($a_{i}\in K,
\gamma_{i}\in \Gamma, 1\leq k_{i}\leq n$ for $i=1,\dots, p$). As it
is shown in \cite[Proposition 2.4.9]{Levin4}, every linear
$\sigma^{\ast}$-ideal in $K\{y_{1},\dots, y_{n}\}^{\ast}$ is prime.
A $\sigma^{\ast}$-polynomial is said to be {\em quasi-linear} if it is
linear with respect to its leader.

\begin{thm}
Let $K$ be an inversive difference field with a basic set $\sigma$
and let $\preccurlyeq$ be a preorder on $K\{y_{1},\dots,
y_{n}\}^{\ast}$ such that $A_{1}\preccurlyeq A_{2}$ if and only if
$u_{A_{2}}$ is a transform of $u_{A_{1}}$ and
$\deg_{u_{A_{1}}}A_{1}\leq\deg_{u_{A_{2}}}A_{2}$. Furthermore, let $A$ be
an irreducible $\sigma^{\ast}$-polynomial in $K\{y_{1},\dots,
y_{n}\}^{\ast}\setminus K$ and $\Gamma A = \{\gamma
A\,|\,\gamma\in\Gamma\}$. Then the set $\mathcal{M}$ of all minimal
(with respect to $\preccurlyeq$) elements of $\Gamma A$ is a
characteristic set of the $\sigma^{\ast}$-ideal $[A]^{\ast}$.
\end{thm}

\begin{proof}
By \cite[Theorem 2.4.13]{Levin4}, if $M$ is a nonzero
$\sigma^{\ast}$-polynomial in $[A]^{\ast}$ written as $M =
\D\sum_{i=1}^{l}C_{i}A_{i}$ \, {\em ($l\geq 1$)}, where $C_{i}\in
K\{y_{1},\dots, y_{n}\}^{\ast}$ {\em ($1\leq i\leq l$)} and $A_{i} =
\gamma_{i}A$ for some distinct elements $\gamma_{1},\dots,
\gamma_{l}\in \Gamma$, then
$\deg_{u_{A_{k}}}M\geq\deg_{u_{A_{k}}}A_{k}$ for some $k$, $1\leq
k\leq l$. It follows that if $\gamma A$ is an element of
$\mathcal{M}$ such that $\gamma A\preccurlyeq A_{k}$, then $M$ is
not reduced with respect to $\gamma A$. Thus, the
$\sigma^{\ast}$-ideal $[A]^{\ast}$ contains no nonzero
$\sigma^{\ast}$-polynomial reduced with respect to $\mathcal{M}$, so
$\mathcal{M}$ is a characteristic set of this ideal.
\end{proof}

\begin{prop}
Let $K$ be an inversive difference field with a basic set $\sigma =
\{\alpha_{1},\dots, \alpha_{m}\}$, $R = K\{y_{1},\dots,
y_{n}\}^{\ast}$ the ring of $\sigma^{\ast}$-polynomials in
$\sigma^{\ast}$-indeterminates $y_{1},\dots, y_{n}$ over $K$, and
$A$ a quasi-linear (not necessarily irreducible)
$\sigma^{\ast}$-polynomial in $R\setminus K$ with leader $u_{A}$.
Furthermore, let $M$ be a nonzero $\sigma^{\ast}$-polynomial in the
ideal $[A]^{\ast}$ of $R$ written in the form $M =
\D\sum_{i=1}^{l}C_{i}A_{i}$ \, {\em ($l\geq 1$)} where $C_{i}\in
K\{y_{1},\dots, y_{n}\}^{\ast}$ {\em ($1\leq i\leq l$)} and $A_{i} =
\gamma_{i}A$ for some distinct elements $\gamma_{1},\dots,
\gamma_{l}\in \Gamma$ such that $\gamma_{i}\sim u_{A}$ for
$i=1,\dots, l$. Finally, let $u_{i}$ denote the leader of the
$\sigma^{\ast}$-polynomial $A_{i}$ {\em ($i=1,\dots, l$)}. Then
there exists $\nu\in\mathbb{N}$, $1\leq \nu\leq l$, such that
$\deg_{u_{\nu}}M\geq 1$.
\end{prop}

\begin{proof}
Note that even though the quasi-linear $\sigma^{\ast}$-polynomial
$A$ is not necessarily irreducible, it is irreducible as a
polynomial in its leader $u_{A}$ over the field of rational
functions in other terms of $A$. It follows that one can use the
arguments of the proof of \cite[Theorem 2.4.13]{Levin4}; this
theorem assumes that $A$ is irreducible, but the proof actually uses
only the fact that $A$ is irreducible as a univariate polynomial in
$u_{A}$ whose coefficients are rational functions of the other terms
of $A$. With this remark, our theorem becomes a consequence of
\cite[Theorem 2.4.13]{Levin4}.
\end{proof}

\begin{prop}
With the notation of the last proposition, let $A$ be a quasi-linear
$\sigma^{\ast}$-polynomial in $K\{y_{1},\dots,
y_{n}\}^{\ast}\setminus K$ of the form $A = au_{A} +B$ where $a\in
K$ and $B\in K\{y_{1},\dots, y_{n}\}^{\ast}$ (all terms of $B$ are
smaller than $u_{A}$). Then the $\sigma^{\ast}$-ideal $[A]^{\ast}$
of $K\{y_{1},\dots, y_{n}\}^{\ast}$ is prime. Furthermore, the set
$\mathcal{M}$ of all minimal (with respect to $\preccurlyeq$)
elements of $\Gamma A$ (we use the notation of Theorem 3.9) is a
characteristic set of $[A]^{\ast}$.
\end{prop}

\begin{proof}
Without loss of generality we can assume that $a=1$. Let
$u_{A}\in(\Gamma Y)_{j}$ ($1\leq j\leq 2^{m}$).

If $FE\in [A]^{\ast}$ ($F, E\in R$), then $FE$ can be written as $FE
= \D\sum_{i=1}^{s}C_{i}\gamma_{i}A$ where $C_{i}\in R$ and
$\gamma_{i}\in\Gamma$ ($1\leq i\leq s$). Then one can take
$\gamma\in\Gamma_{j}$ such that all terms of the
$\sigma^{\ast}$-polynomials $\gamma F$, $\gamma G$, $\gamma C_{i}$
and $\gamma\gamma_{i}A$ ($1\leq i\leq s$) lie in $(\Gamma Y)_{j}$.
Applying $\gamma$ to the last equality we obtain that $(\gamma
F)(\gamma E) = \D\sum_{i=1}^{s}(\gamma C_{i})(\gamma\gamma_{i}A)$.
Thus, without loss of generality we can assume that all terms of
$F$, $E$ and $C_{i}$ ($1\leq i\leq s$) in the representation $FE =
\D\sum_{i=1}^{s}C_{i}\gamma_{i}A$ belong to $(\Gamma Y)_{j}$ and
$\gamma_{i}\sim u_{A}$ for $i=1,\dots, s$. Now one can subtract from
$F$ some linear combination of elements of the form $\beta A$ with
$\beta\in\Gamma_{j}$ to eliminate all transforms of $u_{A}$ in $F$.
We obtain that $F\equiv F_{1}\, \, (mod\,(\{\beta
A\,|\,\beta\in\Gamma_{j}\})\,)$ where $F_{1}$ does not contain any
$\gamma u_{A}$ with $\gamma\in\Gamma_{j}$. Similarly, $E\equiv
E_{1}\, \, (mod\,(\{\beta A\,|\,\beta\in\Gamma_{j}\})\,)$ where
$E_{1}$ does not contain any $\gamma u_{A}$ with
$\gamma\in\Gamma_{j}$. If $F, E\notin [A]^{\ast}$, then $F_{1},
E_{1}\notin [A]^{\ast}$, but $F_{1}E_{1}\in [A]^{\ast}$, since
$FE\in [A]^{\ast}$. At the same time, $F_{1}E_{1}$ does not contain
any transform of $u_{A}$ and can be written as a linear combination
of elements of the form $\beta A$ where $\beta\sim u_{A}$. We get a
contradiction with the statement of Proposition 3.10, so our
proposition is proved. (The last statement is a direct consequence
of Theorem 3.9, since the $\sigma^{\ast}$-polynomial $A$ is
irreducible.)
\end{proof}

The following result, proved in \cite[Theorem 4.2.5]{Levin4}
introduces a Hilbert-type dimension polynomial associated with a
prime $\sigma^{\ast}$-ideal of a ring of
$\sigma^{\ast}$-polynomials.

\begin{thm}
Let $K$ be an inversive difference field with a basic set of
automorphisms $\sigma = \{\alpha_{1},\dots, \alpha_{m}\}$,
$R=K\{y_{1},\dots, y_{n}\}^{\ast}$ the ring of
$\sigma^{\ast}$-polynomials over $K$, and $P$ a prime
$\sigma^{\ast}$-ideal of $R$. Let $L$ denote the quotient field of
$R/P$ treated as the $\sigma^{\ast}$-field extension
$K\langle\eta_{1},\dots,\eta_{n}\rangle^{\ast}$ of $K$ where
$\eta_{i}$ ($1\leq i\leq n$) is the canonical image of $y_{i}$ in
$R/P$. Then there exists a polynomial $\psi_{\eta |
K}(t)\in\mathbb{Q}[t]$ with the following properties.

\smallskip

{\em (i)}\,  $\psi_{\eta | K}(r) = \trdeg_{K}K(\{\gamma \eta_{j} |
\gamma \in \Gamma(r), 1\leq j\leq n\})$  for all sufficiently large
$r\in\mathbb{N}$.

\smallskip

{\em (ii)}\, $\deg\,\psi_{\eta | K}(t)\leq m$ and the polynomial
$\psi_{\eta | K}(t)$ can be written as
\begin{equation}\psi_{\eta | K}(t) = {\frac{2^{m}a}{m!}}t^{m} + o(t^{m})\end{equation}
where $a\in\mathbb{Z}$ and $o(t^{m})$ is a polynomial of degree less
than $m$.

\smallskip

{\em (iii)}\, The integers $a$, $d = deg\,\psi_{\eta | K}(t)$ and
the coefficient of $t^{d}$ in the polynomial $\psi_{\eta | K}(t)$ do
not depend on the choice of a system of $\sigma$-generators $\eta$.
Furthermore, $a = \sigma$-$\trdeg_{K}L$.

\smallskip

{\em (iv)}\, Let $\mathcal{A} = \{A_{1},\dots, A_{p}\}$ be a
characteristic set of the $\sigma^{\ast}$-ideal $P$ and for every
$i=1,\dots, n$, let
$$E_{i} = \{(e_{i1},\dots,
e_{im})\in\mathbb{Z}^{m}\,|\,\alpha_{1}^{e_{i1}}\dots\alpha_{m}^{e_{im}}y_{i}\,\,\,
\text{is the leader of some element of}\,\,\,  \mathcal{A}\}$$ (of
course, some sets $E_{i}$ might be empty). Then
$$\psi_{\eta|K}(t) = \sum_{i=1}^{n}\phi_{E_{i}}(t)$$ where
$\phi_{E_{i}}(t)$ is the dimension polynomial of the set
$E_{i}\subseteq\mathbb{Z}^{m}$ whose existence is established by
Theorem 2.5.
\end{thm}
The polynomial $\psi_{\eta | K}(t)$ whose existence is established
by Theorem 3.12 is called the {\em $\sigma^{\ast}$-dimension
polynomial} of the prime $\sigma^{\ast}$-ideal $P$. The last
statement of Theorem 3.12, together with Theorems 2.6 and 2.4, gives
a method of computation of the $\sigma^{\ast}$-dimension polynomial
associated with a prime $\sigma^{\ast}$-ideal of the ring of
$\sigma^{\ast}$-polynomials $K\{y_{1},\dots, y_{n}\}^{\ast}$.
Therefore, it provides a method of computation of the Einstein's
strength of a prime system of algebraic partial difference
equations. (The $\sigma^{\ast}$-polynomials of such a system
generate a prime $\sigma^{\ast}$-ideal $P$ of $K\{y_{1},\dots,
y_{n}\}^{\ast}$; as it is explained in \cite[Section 7.7]{Levin4},
the Einstein's strength of the system is expressed by the
$\sigma^{\ast}$-dimension polynomial of $P$.) In short, the
$\sigma^{\ast}$-dimension polynomial determines the number of
parameters in the general solution of the system that can be chosen
arbitrarily (the ``arbitrariness'' of the general solution).
Therefore, if two systems adequately describe a process, one should
prefer to work with a system with the smaller
$\sigma^{\ast}$-dimension polynomial. (Such polynomials are compared
with respect to the natural order: $f(t)\leq g(t)$ if $f(r)\leq
g(r)$ for all sufficiently large $r\in\mathbb{N}$.) In the next part
of the paper, the results of Theorem 3.9, Proposition 3.11 and Theorem
3.12 will be used for the evaluation of the strength of systems of
difference equations that represent finite-difference schemes for
PDEs describing certain chemical processes.

\section{Evaluation of the Einstein's strength of difference schemes
for some reaction-diffusion equations}

{\bf 1.}\, {\bf The diffusion equation in one spatial dimension} for
a constant collective diffusion coefficient $a$ and unknown function
$u(x, t)$ describing the density of the diffusing material at given
position $x$ and time $t$ is as follows:
\begin{equation}{\frac{\partial u(x, t)}{\partial t}} =
c{\frac{\partial^{2}u(x, t)}{\partial x^{2}}}.\end{equation} ($c$ is
a constant). Let us compute the strength of difference equations
that arise from three most common difference schemes for equation
(4.1).

\medskip

\large
\begin{center}
Strength of the forward difference scheme
\end{center}
\normalsize
\medskip

The forward difference scheme for the diffusion
equation (4.1) is based on the standard approximations $\D\frac{u(x + h,
t) - u(x, t)}{h}$ and $\D\frac{u(x, t) - u(x, t+h)}{h}$ for ${\D\frac{\partial u(x, t)}{\partial x}}$ and
${\D\frac{\partial u(x, t)}{\partial t}}$, respectively, with a small step $h$.

We obtain the equation in finite differences
\begin{equation}
u(x, t + h) - u(x, t) = a(u(x + 2h, t) - 2u(x + h, t) + u(x, t)).
\end{equation}
where $a = c/h$. ($\D\frac{\partial^{2}u(x, t)}{\partial x^{2}}$ is replaced with $\D\frac{u(x + 2h, t) - 2u(x + h, t) + u(x, t)}{h^{2}}$).

\medskip

Let $K$ be an inversive difference functional field with basic set
$\sigma = \{\alpha_{1}:f(x, t)\mapsto f(x+h, t),\, \alpha_{2}:f(x,
t)\mapsto f(x, t+h)\}$ ($f(x, t)\in K$) containing $a$ and let
$K\{y\}^{\ast}$ be the ring of $\sigma^{\ast}$-polynomials in one
$\sigma^{\ast}$-indeterminate $y$ over $K$. Treating $y$ as the
unknown function $u(x, t)$ in the equation (4.2), we can write this
equation as
\begin{equation}
a\alpha_{1}^{2}y - 2a\alpha_{1}y - \alpha_{2}y + (a+1)y = 0.
\end{equation}
Let $A$ denote the $\sigma^{\ast}$-polynomial in the left-hand side
of the last equation. Since $A$ is linear, it generates a linear
(and therefore a prime) $\sigma^{\ast}$-ideal $P = [A]^{\ast}$ in
$K\{y\}^{\ast}$.

\medskip

Applying Proposition 3.11, we obtain a characteristic set
$\mathcal{A} = \{A_{1}, A_{2}, A_{3}, A_{4}\}$ of the ideal $P$
where
\begin{align*}
A_{1} = A = a\alpha_{1}^{2}y - 2a\alpha_{1}y - \alpha_{2}y +
(a+1)y,\\ A_{2} = \alpha_{1}^{-1}A = -\alpha_{1}^{-1}\alpha_{2}y +
a\alpha_{1} y + (a+1)\alpha_{1}^{-1}y - 2ay,\\ A_{3} =
\alpha_{1}^{-1}\alpha_{2}^{-1}A = a\alpha_{1}\alpha_{2}^{-1}y
+(a+1)\alpha_{1}^{-1}\alpha_{2}^{-1}y - \alpha_{1}^{-1}y -
2a\alpha_{2}^{-1}y,\\A_{4} = \alpha_{1}^{-2}\alpha_{2}^{-1}A =
(a+1)\alpha_{1}^{-2}\alpha_{2}^{-1}y - \alpha_{1}^{-2}y
-2a\alpha_{1}\alpha_{2}^{-1}y + a\alpha_{2}^{-1}y.
\end{align*}

\smallskip

The leaders of these $\sigma^{\ast}$-polynomials are
$\alpha_{1}^{2}y,\, \alpha_{1}^{-1}\alpha_{2}y,\,
\alpha_{1}\alpha_{2}^{-1}y$, and $\alpha_{1}^{-2}\alpha_{2}^{-1}y$,
respectively (they are written first in the
$\sigma^{\ast}$-polynomials $A_{i}$ above). Therefore, the
$\sigma^{\ast}$-dimension polynomial of equation (4.3) is equal to
the dimension polynomial of the subset $$E = \{(2, 0), (-1, 1), (1,
-1), (-2, -1)\}$$ of $\mathbb{Z}^{2}$. Applying the results of
theorems 2.6 and 2.4 we obtain that the $\sigma^{\ast}$-dimension
polynomial of equation (4.3) that expresses the Einstein's strength
of the forward difference scheme for (4.1) is
$$\psi_{Forw}(t) = 5t.$$ Note that in this case $$V_{E}(r) = \{(1,
y_{1})\,|\, 0\leq y_{1}\leq r-1\}\bigcup\{(0, y_{2})\,|\,-r\leq
y_{2}\leq r\}\bigcup\{(-1, y_{3})\,|\,-(r-1)\leq y_{3}\leq 0\}$$
$\bigcup\{(x, 0)\,|\,-r\leq x\leq -2\}$ for any $r\in\mathbb{N}$
($y_{1}, y_{2}, y_{3}, x\in\mathbb{Z}$).

\bigskip

\large
\begin{center}
Strength of the symmetric difference scheme
\end{center}
\normalsize \bigskip

Consider the symmetric difference scheme for the diffusion equation (4.1)
obtained by replacing the partial derivatives $\D\frac{\partial^{2}
u(x,t)}{\partial x^{2}}$ and $\D\frac{\partial u(x,t)}{\partial t}$
with $\D\frac{u(x + h, t) - 2u(x, t) + u(x-h, t)}{h^{2}}$ and $\D\frac{u(x, t + h) - u(x,
t-h)}{2h}$, respectively. It leads to the equation in finite differences
\begin{equation}
u(x, t+h) - u(x, t-h) = a(u(x+h, t) - 2u(x, t) + u(x-h, t))
\end{equation}
where $a = 2c/h$. As in the case of the forward difference
scheme, let $K$ be an inversive difference functional field with
basic set $\sigma = \{\alpha_{1}:f(x, t)\mapsto f(x+h, t),\,
\alpha_{2}:f(x, t)\mapsto f(x, t+h)\}$ ($f(x, t)\in K$) and let
$K\{y\}^{\ast}$ be the ring of $\sigma^{\ast}$-polynomials in one
$\sigma^{\ast}$-indeterminate $y$ over $K$ ($y$ is treated as the
unknown function $u(x, t)$;  we also assume that $a\in K$). Then the
equation (4.4) can be written as
\begin{equation}
a\alpha_{1}y + a\alpha_{1}^{-1}y - \alpha_{2}y  + \alpha_{2}^{-1}y -
2ay = 0.
\end{equation}
By Proposition 3.11, the characteristic set of the
$\sigma^{\ast}$-ideal generated by the $\sigma^{\ast}$-polynomial $B
= a\alpha_{1}y + a\alpha_{1}^{-1}y - \alpha_{2}y  + \alpha_{2}^{-1}y
- 2ay$ is $\{B, \alpha_{1}^{-1}B\}$. The leaders of $B$ and
$\alpha_{1}^{-1}B$ are $\alpha_{1}y$ and $\alpha_{1}^{-2}y$,
respectively. Now Theorem 3.12 shows that the strength of
the equation (4.5) is expressed by the dimension polynomial
$\phi_{E}(t)$ of the set $E = \{(1, 0), (-2,
0)\}\subseteq\mathbb{Z}^{2}$ (see Theorem 2.5). By Theorem 2.6, this
polynomial coincides with the dimension polynomial $\omega_{E'}(t)$
of the set $$E' = \{(1, 0, 0, 0), (0, 0, 2, 0), (1, 0, 1, 0), (0, 1,
0, 1)\}\subseteq\mathbb{N}^{4}.$$ Applying formula (2.3) we obtain
that the strength of the equation (4.5), which expresses the
symmetric difference scheme for (4.1), is represented by the
$\sigma^{\ast}$-dimension polynomial
$$\psi_{Symm}(t) = 4t.$$

\medskip

\large
\begin{center}
Strength of the Crank-Nicholson scheme
\end{center}
\normalsize
\medskip

The Crank-Nicholson scheme (see \cite[Section 4]{CN}) applied to the
diffusion equation with the above interpretation of the shifts of
arguments as two automorphisms $\alpha_{1}$ and $\alpha_{2}$ gives
the algebraic difference equation of the form
\begin{equation}
\alpha_{1}\alpha_{2}y + a_{1}\alpha_{1}^{-1}\alpha_{2}y +
a_{2}\alpha_{1}y + a_{3}\alpha_{2}y + a_{4}\alpha_{1}^{-1}y +
a_{5}=0
\end{equation}
where $a_{i}$ ($1\leq i\leq 5$) are constants. Applying Proposition
3.11, we obtain that the $\sigma^{\ast}$-polynomial\, $C =
\alpha_{1}\alpha_{2}y + a_{1}\alpha_{1}^{-1}\alpha_{2}y +
a_{2}\alpha_{1}y + a_{3}\alpha_{2}y + a_{4}\alpha_{1}^{-1}y + a_{5}$
in the left-hand side of the last equation generates a prime
$\sigma^{\ast}$-ideal of $K\{y\}^{\ast}$ whose characteristic set
consists of the $\sigma^{\ast}$-polynomials $C,\,
\alpha_{1}^{-1}C,\, \alpha_{2}^{-1}C$, and
$\alpha_{1}^{-1}\alpha_{2}^{-1}C$. Their leaders are
$\alpha_{1}\alpha_{2}y$, $\alpha_{1}^{-2}\alpha_{2}y$,
$\alpha_{1}\alpha_{2}^{-1}y$, and $\alpha_{1}^{-2}\alpha_{2}^{-1}y$,
respectively. Applying theorems 2.6 and 2.4 to the set $\{(1, 1),
(-2, 1), (1, -1), (-2, -1)\}\subseteq\mathbb{Z}^{2}$ we obtain that
the strength of the equation (4.6) is expressed by the dimension
polynomial
$$\psi_{Crank-Nickolson}(t) = 6t-1 .$$ Thus, the symmetric
difference scheme for the diffusion equation has higher strength
(that is, smaller dimension polynomial) than the forward difference
scheme and the Crank-Nicholson scheme, so the symmetric scheme is
the best among these three schemes from the point of view of the
Einstein's strength.

\bigskip

{\bf 2.}\, {\bf Murray, Fisher, Burger and some other quasi-linear
reaction-diffusion equations.}
\smallskip

Proposition 3.11 allows us to compute the strength of
reaction-diffusion equations of the  form
\begin{equation}
{\frac{\partial u}{\partial t}}  - \frac{\partial^{2}u}{\partial
x^{2}} = H\left(u, {\frac{\partial u}{\partial x}}\right)
\end{equation}
where $u=u(x, t)$ is a function of space and time variables $x$ and
$t$, respectively, and $H\left(u, {\D\frac{\partial u}{\partial
x}}\right)$ is a nonlinear function of $u$ and ${\D\frac{\partial
u}{\partial x}}$. Such equations have recently attracted a lot of
attention in the context of chemical kinetics, mathematical biology
and turbulence. The following PDEs, that are particular cases of
equation (4.7), are in the core of the corresponding mathematical
models.

Murray equation (\cite{Cherniha}, equation (4)):

\begin{equation}
\frac{\partial^{2}u}{\partial x^{2}} + \mu_{1}u{\frac{\partial
u}{\partial t}} + \mu_{2}u - \mu_{3}u^{2} = 0,\,\,\, (\mu_{1},
\mu_{2}, \mu_{3}\,\,\,\text{are constants)}.
\end{equation}

Burgers equation (\cite[Section 17.1, (17.3)]{Wazwaz1}):
\begin{equation}
\frac{\partial^{2}u}{\partial x^{2}} - u{\frac{\partial u}{\partial
x}} - {\frac{\partial u}{\partial t}} = 0.
\end{equation}

Fisher equation (\cite[Section 17.1, (17.4)]{Wazwaz1}):
\begin{equation}
\frac{\partial^{2}u}{\partial x^{2}} - u{\frac{\partial u}{\partial
t}} + u(1-u) = 0.
\end{equation}

Huxley equation (\cite[Section 17.1, (17.5)]{Wazwaz1}):
\begin{equation}
\frac{\partial^{2}u}{\partial x^{2}} - u{\frac{\partial u}{\partial
t}} - u(k-u)(u-1) = 0, \,\,\, k\neq 0.
\end{equation}

Burgers-Huxley equation (\cite[Section 17.1, (17.7)]{Wazwaz1}):

\begin{equation}
\frac{\partial^{2}u}{\partial x^{2}} + u{\frac{\partial u}{\partial
x}} - {\frac{\partial u}{\partial t}} + u(k-u)(u-1)= 0,\,\,\, k\neq
0.
\end{equation}

FitzHugh-Nagumo equation (\cite[Section 17.1, (17.8)]{Wazwaz1}):
\begin{equation}
\frac{\partial^{2}u}{\partial x^{2}} + u{\frac{\partial u}{\partial
x}} + u(1-u)(a-u)= 0, \,\,\, a\neq 0.
\end{equation}

The last six equations are of the form
\begin{equation}
\frac{\partial^{2}u}{\partial x^{2}} + (au + b){\frac{\partial
u}{\partial x}} + c{\frac{\partial u}{\partial t}} + F(u) = 0
\end{equation}
where $a, b, c$ are constants ($c\neq 0$, $ab\neq 0$) and $F(u)$ is
a polynomial in one variable $u$ with coefficients in the ground
functional field $K$. Therefore, the forward difference scheme for
equations (4.8) - (4.13) leads to algebraic difference equations of
the form
\begin{equation}
\alpha_{1}^{2}y + (ay+b-2)\alpha_{1}y + c\alpha_{2}y + G(y) = 0.
\end{equation}
(As before, we set $y=u$, denote the automorphisms of the ground
field $f(x, t)\mapsto f(x+1, t)$ and $f(x, t)\mapsto f(x, t+1)$ by
$\alpha_{1}$ and $\alpha_{2}$, respectively, and write the monomials
in the left-hand side of the equation in the decreasing order of
their highest terms. We also set $G(y) = F(y)-ay^{2}-(b+c-1)y$.)

Applying Proposition 3.11 we obtain that the
$\sigma^{\ast}$-polynomial $A = \alpha_{1}^{2}y +
(ay+b-2)\alpha_{1}y + c\alpha_{2}y + G(y)$ generates a prime
$\sigma^{\ast}$-ideal of $K\{y\}^{\ast}$ ($\sigma = \{\alpha_{1},
\alpha_{2}\}$). As in the case of equation (4.3), we obtain that the
characteristic set of the ideal $[A]^{\ast}$ consists of the
$\sigma^{\ast}$-polynomials $A$, $\alpha_{1}^{-1}A$,
$\alpha_{1}^{-1}\alpha_{2}^{-1}A$ and
$\alpha_{1}^{-2}\alpha_{2}^{-1}A$  with leaders $\alpha_{1}^{2}y$,
$\alpha_{1}^{-1}\alpha_{2}y$, $\alpha_{1}\alpha_{2}^{-1}y$ and
$\alpha_{1}^{-2}\alpha_{2}^{-1}y$, respectively. Therefore (as in
the case of equation (4.3)\,) the $\sigma^{\ast}$-dimension
polynomial that expresses the Einstein's strength of the forward
difference scheme for each of the equations (4.8) - (4.13) is equal
to the dimension polynomial of the set $\{(2, 0), (-1, 1), (1, -1),
(-2, -1)\}\subseteq\mathbb{Z}^{2}$, that is, $$\psi_{Forw}(t) =
5t.$$

\medskip

The symmetric difference scheme for equation (4.15) (and therefore
for each of the equations (4.8) - (4.13)\,) gives an algebraic
difference equation of the form

\begin{equation}
(ay+b+1)\alpha_{1}y + (1-ay-b)\alpha_{1}^{-1}y + c\alpha_{2}y -
c\alpha_{2}^{-1}y + F(y) = 0.
\end{equation}
(Recall that we replace $\D\frac{\partial u}{\partial x}$,
$\D\frac{\partial^{2}u}{\partial x^{2}}$ and $\D\frac{\partial
u}{\partial t}$ with $(\alpha_{1} + \alpha_{1}^{-1} -2)u$,
$(\alpha_{1} - \alpha_{1}^{-1})u$ and $(\alpha_{2} -
\alpha_{2}^{-1})u$, respectively.) If we consider the standard
ranking defined after Definition 3.2, then the quasi-linear
$\sigma^{\ast}$-polynomial in the left-hand side of the equation
(4.16) is not of the form considered in Proposition 3.11. However,
if one considers a similar ranking with $\alpha_{2} > \alpha_{1}$,
then the $\sigma^{\ast}$-polynomial $B = (ay+b+1)\alpha_{1}y +
(1-ay-b)\alpha_{1}^{-1}y + c\alpha_{2}y - c\alpha_{2}^{-1}y + F(y)$
in the left-hand side of (4.16) is a quasi-linear one with the
leader $\alpha_{2}y$. By Proposition 3.11, the
$\sigma^{\ast}$-polynomials $B$ and $\alpha_{2}^{-1}B$ form a
characteristic set of the prime $\sigma^{\ast}$-ideal $[B]^{\ast}$
of $K\{y\}^{\ast}$. Since their leaders are, respectively,
$\alpha_{2}y$ and $\alpha_{2}^{-2}y$, the Einstein's strength of the
symmetric difference scheme for each of the equations (4.8) - (4.13)
is expressed by the dimension polynomial $\psi_{Symm}(t)$ of the set
$\{(1, 0), (0, 2)\}\subseteq\mathbb{Z}^{2}$. As in the case of
equation (4.5), we obtain that  $$\psi_{Symm}(t) = 4t.$$

Thus, one should prefer the symmetric scheme to the forward one
while considering the Einstein's strength of these schemes for  PDEs
(4.8) - (4.13).

\bigskip

{\bf 3.}\, {\bf The mathematical model of chemical reaction kinetics
with the diffusion phenomena} is described by a system of partial
differential equations of the form

\begin{equation}
\left \{
\begin{array}{lcl}
{\D\frac{\partial u_{1}}{\partial t}} =  {\D\frac{\partial^{2} u_{1}}{\partial x^{2}}}\\\\
{\D\frac{\partial u_{2}}{\partial t}} =  {\D\frac{\partial^{2} u_{2}}{\partial x^{2}}},\\\\
{\D\frac{\partial u_{1}}{\partial t}} =  {\D\frac{\partial^{2}
u_{3}}{\partial x^{2}}} - k_{1}u_{3}^{2} + k_{1}u_{3}u_{1} +
k_{2}u_{2} -k_{2}u_{3}.
\end{array}\right.
\end{equation}
(see \cite{Lim}).

\medskip

The forward difference scheme leads to the following system of
algebraic difference equations with three
$\sigma^{\ast}$-indeterminates  $y_{1},\, y_{2}$ and $y_{3}$ (they
stand for $u_{1},\, u_{2}$ and $u_{3}$, respectively), where $\sigma
= \{\alpha_{1}:f(x, t)\mapsto f(x+1, t),\, \alpha_{2}:f(x, t)\mapsto
f(x, t+1)\}$ ($f(x, t)$ is an element of the inversive ground
functional field $K$).

\begin{equation}
\left \{
\begin{array}{lcl}
\alpha_{1}^{2}y_{1} - 2\alpha_{1}y_{1} - \alpha_{2}y_{1} + 2y_{1} =
0,\\\\
\alpha_{1}^{2}y_{2} - 2\alpha_{1}y_{2} - \alpha_{2}y_{2} + 2y_{2} =
0,\\\\
\alpha_{1}^{2}y_{3} - 2\alpha_{1}y_{3} - \alpha_{2}y_{3} +
k_{1}y_{1}y_{3} - k_{1}y_{3}^{2} + k_{2}y_{2} - k_{2}y_{3} = 0
\end{array}
\right.
\end{equation}
where $k_{1}, k_{2}, k_{3}$ are constants in $K$.

Let $A_{1}$, $A_{2}$, and $A_{3}$ be the $\sigma^{\ast}$-polynomials
in the left-hand sides of the  first, second and third equations of
the last system, respectively. Combining Proposition 2.4.9 of
\cite{Levin4} (that states that every linear $\sigma^{\ast}$-ideal
in a ring of $\sigma^{\ast}$-polynomials is prime) and our
Proposition 3.11 we obtain that the $\sigma^{\ast}$-ideal $[A_{1},
A_{2}, A_{3}]^{\ast}$ of the ring $K\{y_{1}, y_{2}, y_{3}\}^{\ast}$
is prime. Since $A_{1}$ and $A_{2}$ are linear
$\sigma^{\ast}$-polynomials in different
$\sigma^{\ast}$-indeterminates ($y_{1}$ and $y_{2}$, respectively)
and $A_{3}$ is a quasi-linear $\sigma^{\ast}$-polynomial with
coefficient $1$ of its leader $\alpha_{1}^{2}y_{3}$ (it can be
treated as a quasi-linear $\sigma^{\ast}$-polynomial in $y_{3}$ over
the quotient $\sigma^{\ast}$-field of $K\{y_{1}, y_{2},
y_{3}\}^{\ast}/[A_{1}, A_{2}]^{\ast}$), one can apply Proposition
3.11 to obtain that the twelve $\sigma^{\ast}$-polynomials $A_{i}$,
$\alpha_{1}^{-1}A_{i}$, $\alpha_{1}^{-1}\alpha_{2}^{-1}A_{i}$,
$\alpha_{1}^{-2}\alpha_{2}^{-1}A_{i}$ ($i=1, 2, 3$) form a
characteristic set of $[A_{1}, A_{2}, A_{3}]^{\ast}$ (cf. the
characteristic set of the $\sigma^{\ast}$-ideal generated by the
left-hand side of equation (4.3)\,). Proceeding as in the case of
forward difference scheme of the diffusion equation we obtain that
the strength of the forward difference scheme for system (4.17) is
expressed by the polynomial
$$\psi_{Forw}(t) = 15t.$$

Using the above arguments and the results for difference schemes for
equation (4.1), we obtain that the strengths of the symmetric and
Crank-Nicholson schemes for (4.17) are expressed with the
polynomials
$$\psi_{Symm}(t) = 12t\,\,\,
\text{and}\,\,\, \psi_{Crank-Nicholson}(t) = 18t - 3,$$
respectively. Therefore, in our case, as in the case of equation
(4.1), the symmetric scheme for system (4.17) is characterized by
the smallest $\sigma^{\ast}$-dimension polynomial (and therefore by
the highest Einstein's strength) among these schemes.

\end{document}